\documentclass[twoside,a4paper,11pt]{amsart}

\input xypic
\usepackage{amssymb,enumitem}
\xyoption{all}

\newtheorem{thm}{Theorem}[section]

\newtheorem{prop}[thm]{Proposition}
\newtheorem{rem}[thm]{Remark}
\newtheorem{defn}[thm]{Definition}

\author{Jeremy Rickard} 
\address{School of Mathematics, University of Bristol, Bristol BS8 1TW, UK}
\email{j.rickard@bristol.ac.uk}
\date\today 
\title{A cocomplete but not complete abelian category} 

\begin{document}
\maketitle

\begin{abstract}
  An example of a cocomplete abelian category that is not complete is
  constructed.
\end{abstract}

\section{Introduction}
\label{sec:intro}

In this paper we answer a question that was posted on the internet
site MathOverflow by Simone Virili~\cite{MO}. The question asked
whether there is an abelian category that is cocomplete but not
complete. It seemed that there must be a standard example, or an easy
answer, but despite receiving a fair amount of knowledgeable interest,
after many months the question still had no solution.

Recall that a complete category is one with all small limits, and that
for an abelian category, which by definition has kernels, completeness
is equivalent to the existence of all small products. Similarly,
because of the existence of cokernels, cocompleteness for abelian
categories is equivalent to the existence of all small coproducts.

A common, but wrong, first reaction to the question is that there are
easy natural examples, with something such as the category of torsion
abelian groups coming to mind: an infinite direct sum of torsion
groups is torsion, but an infinite direct product of torsion groups
may not be. Nevertheless, this category does have products: the
product of a set of groups is simply the torsion subgroup of their
direct product. In categorical terms, the category of torsion abelian
groups is a coreflective subcategory of the category of all abelian
groups (which certainly has products), with the functor sending a
group to its torsion subgroup being right adjoint to the inclusion
functor.

Similar, but more sophisticated, considerations doom many approaches
to finding an example.

Recall that an $\mathtt{(AB5)}$ category, in the terminology of
Grothendieck~\cite{grothendieck}, is a cocomplete abelian category in
which filtered colimits are exact, and that a Grothendieck category is
an $\mathtt{(AB5)}$ abelian category with a generator. The favourite
cocomplete abelian category of the typical person tends to be a
Grothendieck category, such as a module category or a category of
sheaves, but it is well known that every Grothendieck category is
complete (for a proof see, for
example,~\cite[Prop. 8.3.27]{kashiwara_schapira}).

More generally, any locally presentable category is
complete~\cite[Cor. 1.28]{adamek_rosicki}, and most standard
constructions of categories preserve local presentability.

\section{The construction}
\label{sec:main}

First we shall fix a chain of fields
$\{k_\alpha\mid\alpha\in\mathbf{On}\}$ indexed by the ordinals such
that $k_\beta/k_\alpha$ is an infinite degree field extension whenever
$\alpha<\beta$. For example, we could take a class
$\{x_\alpha\mid\alpha\in\mathbf{On}\}$ of variables indexed by
$\mathbf{On}$ and let $k_\alpha=\mathbb{Q}(X_{<\alpha})$ be the field
of rational functions in the set of variables
$X_{<\alpha}=\{x_\gamma\mid\gamma<\alpha\}$.

We generally adopt the convention that categories are locally small
(i.e., the class of morphisms between two objects is always a
set). However, we'll start by defining a ``category'' $\mathbf{C}$
which is not locally small.

An object $V$ of $\mathbf{C}$ consists of a
$k_\alpha$-vector space $V_\alpha$ for each ordinal $\alpha$, together
with a $k_\alpha$-linear map $v_{\alpha,\beta}:V_\alpha\to V_\beta$
for each pair $\alpha<\beta$ of ordinals, such that
$v_{\alpha,\gamma}=v_{\beta,\gamma}v_{\alpha,\beta}$ whenever
$\alpha<\beta<\gamma$. (When we denote an object by an upper case
letter such as $V$, we will always use, without further comment, the
corresponding lower case letter for the linear maps
$v_{\alpha,\beta}$.)

A morphism $\varphi:V\to W$ of $\mathbf{C}$ consists of a
$k_\alpha$-linear map $\varphi_\alpha:V_\alpha\to W_\alpha$ for each
ordinal $\alpha$, such that
$\varphi_\beta v_{\alpha,\beta}=w_{\alpha,\beta}\varphi_\alpha$
whenever $\alpha<\beta$.

Composition is defined in the obvious way. If $\vartheta:U\to V$ and
$\varphi:V\to W$ are morphisms, then
$(\varphi\vartheta)_\alpha=\varphi_\alpha\vartheta_\alpha$.

It is straightforward to check that $\mathbf{C}$ is an additive (in
fact, $k_0$-linear) category, but not locally small: for example, if
$V$ is the object with $V_\alpha=k_\alpha$ for every $\alpha$ and
$v_{\alpha,\beta}=0$ for all $\alpha<\beta$, then a morphism
$\varphi:V\to V$ has $v_\alpha:k_\alpha\to k_\alpha$ multiplication by
some scalar $\lambda_\alpha\in k_\alpha$ for some arbitrary choice of
$\left\{\lambda_\alpha\mid\alpha\in\mathbf{On}\right\}$,
so the class of endomorphisms of $V$ is a proper class.

\begin{prop}
  $\mathbf{C}$ is a (not locally small) abelian category
  with (small) products and coproducts in which (small) filtered
  colimits are exact.
\end{prop}

\begin{proof}
  It is straightforward to check that the obvious ``pointwise''
  constructions give kernels, cokernels, products and coproducts. For
  example, if $\varphi:V\to W$ is a morphism, then the kernel of
  $\varphi$ is the object $U$ with $U_\alpha$ the kernel of
  $\varphi_\alpha:V_\alpha\to W_\alpha$ and
  $u_{\alpha,\beta}:U_\alpha\to U_\beta$ the natural map between the
  kernels of $\varphi_\alpha$ and $\varphi_\beta$ induced by
  $v_{\alpha,\beta}$. 

  Since all these constructions are ``pointwise'', it is also
  straightforward to check that every monomorphism is the kernel of
  its cokernel and every epimorphism is the cokernel of its kernel (so
  that $\mathbf{C}$ is abelian), and that small filtered colimits are
  exact, as the verification reduces to the corresponding facts for
  the category of $k_\alpha$-vector spaces.
\end{proof}

The (locally small) category that we're really interested in is a full
subcategory $\mathbf{G}$ of $\mathbf{C}$.

\begin{defn}
  An object $V$ of $\mathbf{C}$ is {\bf $\alpha$-grounded} if, for
  every $\beta>\alpha$, $V_\beta$ is generated as a $k_\beta$-vector
  space by the image of $v_{\alpha,\beta}$. The full subcategory of
  $\mathbf{C}$ consisting of the $\alpha$-grounded objects is denoted
  by $\alpha$-$\mathbf{G}$.
\end{defn}

\begin{defn}
  An object $V$ of $\mathbf{C}$ is {\bf grounded} if it is
  $\alpha$-grounded for some ordinal $\alpha$. The full subcategory of
  $\mathbf{C}$ consisting of the grounded objects is denoted by
  $\mathbf{G}$.
\end{defn}

\begin{thm}
  $\mathbf{G}$ is a (locally small) $\mathtt{(AB5)}$ abelian category
  that is not complete.
\end{thm}

\begin{proof}
  Let $V$ be an $\alpha$-grounded object. If $\varphi: V\to W$ is a
  morphism then, for $\beta>\alpha$, $\varphi_\beta$ is determined by
  $\varphi_\alpha$, and so the class of morphisms $\varphi: V\to W$ is
  a set. Hence $\mathbf{G}$ is locally small.

  Let $\varphi:V\to W$ be a morphism between $\alpha$-grounded
  objects. Clearly the kernel and cokernel of $\varphi$ are also
  $\alpha$-grounded. Hence $\mathbf{G}$ is closed under kernels and
  cokernels in $\mathbf{C}$ and so it is an exact abelian
  subcategory of $\mathbf{C}$.

  If $\left\{V^i\mid i\in I\right\}$ is a set of grounded objects,
  then there is some ordinal $\alpha$ so that every $V^i$ is
  $\alpha$-grounded. Then $\bigoplus_{i\in I}V^i$ is also
  $\alpha$-grounded. Thus $\mathbf{G}$ is cocomplete and the inclusion
  functor $\mathbf{G}\hookrightarrow\mathbf{C}$ preserves coproducts,
  and hence all colimits. Exactness of filtered colimits in
  $\mathbf{G}$ therefore follows from the same property for
  $\mathbf{C}$.

  Thus $\mathbf{G}$ is a locally small $\mathtt{(AB5)}$ abelian
  category.

  For an ordinal $\alpha$, let $M^\alpha$ be the object with
$$M^\alpha_\beta=
\begin{cases}
  0 &\mbox{if }\beta<\alpha\\
k_\beta&\mbox{if }\beta\geq\alpha,
\end{cases}
$$
and $m^\alpha_{\beta,\gamma}:k_\beta\to k_\gamma$ the inclusion map
for $\alpha\leq\beta<\gamma$. Then $M^\alpha$ is $\alpha$-grounded.

If $W$ is any object of $\mathbf{C}$ then a morphism
$\varphi:M^\alpha\to W$ is determined by
$\varphi_\alpha: M^\alpha_\alpha=k_\alpha\to W_\alpha$, so
$\mathbf{C}(M^\alpha,W)\cong W_\alpha$ and $M^\alpha$ represents the
functor $W\mapsto W_\alpha$ from $\mathbf{C}$ to $k_\alpha$-vector
spaces. Also, if $\alpha<\beta$ then the obvious morphism
$\varphi:M^\beta\to M^\alpha$, with $\varphi_\gamma$ the identity map
$k_\gamma\to k_\gamma$ for $\gamma\geq\beta$, induces a commutative
square
$$\xymatrix{
  \mathbf{C}(M^\alpha,W)\ar[r]^-*[@]{\sim}\ar[d]_{\mathbf{C}(\varphi,W)}
  &W_\alpha\ar[d]^{w_{\alpha,\beta}}\\
  \mathbf{C}(M^\beta,W)\ar[r]^-*[@]{\sim}&W_\beta\\
}$$

We shall show that in $\mathbf{G}$ there is no product of a countably
infinite collection of copies of $M^\alpha$, so $\mathbf{G}$ is not
complete.

Suppose that $P$ is the product (in $\mathbf{G}$) of countably many
copies of $M^\alpha$. Then for any ordinal $\beta\geq\alpha$,
$$\mathbf{G}(M^\beta,P)\cong\mathbf{G}(M^\beta,M^\alpha)^\mathbb{N}
\cong (M^\alpha_\beta)^\mathbb{N}\cong k_\beta^\mathbb{N}$$ as
$k_\beta$-vector spaces.

But if $\alpha<\beta<\gamma$ then $k_\gamma^\mathbb{N}$ is not
generated as a $k_\gamma$-vector space by $k_\beta^\mathbb{N}$,
since $k_\gamma/k_\beta$ is an infinite field extension. So $P$ cannot
be $\beta$-grounded for any $\beta$.
\end{proof}

\begin{rem}
  We have already noted that products do exist in $\mathbf{C}$, with
  the obvious pointwise construction. However, the product of
  $\alpha$-grounded objects need not be $\beta$-grounded for any
  $\beta$.

  Products also exist in $\alpha$-$\mathbf{G}$, since although the
  pointwise product may not be $\alpha$-grounded,
  $\alpha$-$\mathbf{G}$ is a coreflective subcategory of
  $\mathbf{C}$. We can $\alpha$-ground an object $V$ of $\mathbf{C}$
  by replacing $V_\beta$ by $v_{\alpha,\beta}(V_\alpha)$ for
  $\beta>\alpha$, and the product of a set of objects in
  $\alpha$-$\mathbf{G}$ is obtained by $\alpha$-grounding the product
  in $\mathbf{C}$.
\end{rem}

\begin{rem}
  Each category $\alpha$-$\mathbf{G}$ is a cocomplete and complete
  abelian category (in fact, a Grothendieck category, with
  $\bigoplus_{\beta\leq\alpha}M^\beta$ a generator). The inclusion
  functors $\alpha$-$\mathbf{G}\to\beta$-$\mathbf{G}$ are exact and
  preserve coproducts, but do not preserve products, which explains
  why their union $\mathbf{G}$ has coproducts but does not have
  products (or at least not in an obvious way). Thanks to Zhen Lin Low
  for making this observation in a comment on MathOverflow~\cite{MO}.
\end{rem}

\begin{rem}
  Of course, an example of a complete abelian category that is not
  cocomplete can be constructed by taking the opposite category of
  $\mathbf{G}$.
\end{rem}

\section*{Acknowledgements}
I would like to thank MathOverflow and its community for introducing
me to this and many other interesting questions, and for their
comments on this question in particular.

\bibliography{mybib}{} 

\providecommand{\bysame}{\leavevmode\hbox to3em{\hrulefill}\thinspace}
\providecommand{\MR}{\relax\ifhmode\unskip\space\fi MR }
\providecommand{\MRhref}[2]{%
  \href{http://www.ams.org/mathscinet-getitem?mr=#1}{#2}
}
\providecommand{\href}[2]{#2}
\begin{thebibliography}{Gro57}

\bibitem[AR94]{adamek_rosicki}
Ji\v{r}\'i Ad\'amek and Ji\v{r}\'i Rosick\'y, \emph{Locally presentable and
  accessible categories}, London Mathematical Society Lecture Note Series, vol.
  189, Cambridge University Press, Cambridge, 1994.

\bibitem[Gro57]{grothendieck}
Alexander Grothendieck, \emph{Sur quelques points d'alg\`ebre homologique},
  T\^ohoku Math. J. (2) \textbf{9} (1957), 119--221.

\bibitem[KS06]{kashiwara_schapira}
Masaki Kashiwara and Pierre Schapira, \emph{Categories and sheaves},
  Grundlehren der Mathematischen Wissenschaften [Fundamental Principles of
  Mathematical Sciences], vol. 332, Springer-Verlag, Berlin, 2006.

\bibitem[Vir12]{MO}
Simone Virili, \emph{Cocomplete but not complete abelian category},
  MathOverflow, 2012, URL:https://mathoverflow.net/q/112574.

\end{thebibliography}
\bibliographystyle{amsalpha}

\end{document}